\pdfoutput=1
\RequirePackage{ifpdf}
\ifpdf 
\documentclass[pdftex]{sigma}
\else
\documentclass{sigma}
\fi

\numberwithin{equation}{section}

\newtheorem{Theorem}{Theorem}[section]
\newtheorem{Corollary}[Theorem]{Corollary}
\newtheorem{Lemma}[Theorem]{Lemma}

 { \theoremstyle{definition}
\newtheorem{Definition}[Theorem]{Definition}

\newtheorem{Remark}[Theorem]{Remark} }

\begin{document}
\allowdisplaybreaks

\newcommand{\arXivNumber}{2006.04092}

\renewcommand{\thefootnote}{}

\renewcommand{\PaperNumber}{114}

\FirstPageHeading

\ShortArticleName{The Measure Preserving Isometry Groups of Metric Measure Spaces}

\ArticleName{The Measure Preserving Isometry Groups\\ of Metric Measure Spaces\footnote{This paper is a~contribution to the Special Issue on Scalar and Ricci Curvature in honor of Misha Gromov on his 75th Birthday. The full collection is available at \href{https://www.emis.de/journals/SIGMA/Gromov.html}{https://www.emis.de/journals/SIGMA/Gromov.html}}}

\Author{Yifan GUO~$^{\dag\ddag}$}

\AuthorNameForHeading{Y.~Guo}

\Address{$^\dag$~Beijing Institute of Mathematical Sciences and Applications, Beijing, P.R.~China}
\EmailD{\href{mailto:yifan_guo@foxmail.com}{yifan\_guo@foxmail.com}}
\Address{$^\ddag$~Department of Mathematics, University of California, Irvine, CA, USA}

\ArticleDates{Received June 30, 2020, in final form November 02, 2020; Published online November 10, 2020}

\Abstract{Bochner's theorem says that if $M$ is a compact Riemannian manifold with negative Ricci curvature, then the isometry group $\operatorname{Iso}(M)$ is finite. In this article, we show that if $(X,d,m)$ is a compact metric measure space with synthetic negative Ricci curvature in Sturm's sense, then the measure preserving isometry group $\operatorname{Iso}(X,d,m)$ is finite. We also give an effective estimate on the order of the measure preserving isometry group for a~compact weighted Riemannian manifold with negative Bakry--\'Emery Ricci curvature except for small portions.}


\Classification{53C20; 53C21; 53C23}

\renewcommand{\thefootnote}{\arabic{footnote}}
\setcounter{footnote}{0}

\section{Introduction}

Throughout this article, a metric measure space $(X,d,m)$ means that $(X,d)$ is a complete separable metric space, and $m$ is a Borel $\sigma$-finite measure on $X$, which is also finite on bounded sets.
We denote by $\operatorname{Iso}(X,d,m)$ the group of isometries of $(X,d)$ which preserve the measure~$m$. We also denote by $\#\operatorname{Iso}(X,d,m)$ the number of elements in $\operatorname{Iso}(X,d,m)$.

An example of metric measure space is a weighted Riemannian manifold $(M,g,m)$, which means that $(M,g)$ is an $n$-dimensional Riemannian manifold, and $m={\rm e}^{-v}\mathrm {vol}_g$ where $\mathrm {vol}_g$ is the volume element associated with $g$ and $v\colon M\to\mathbb R$ is a $C^3$ function. In this case, $X=M$, $d=d_g$ is the intrinsic metric induced by $g$ and $m={\rm e}^{-v}\mathrm {vol}_g$.
For $N\in [n,\infty)$ we define the {$N$-Bakry--\'Emery Ricci tensor} on $(M,g,m)$ by
\[{\operatorname{Ric}}_{N,m}:={\operatorname{Ric}}+\nabla^2v-\frac{ dv\otimes dv}{N-n}\]
with the convention that when $N=n$, we require $v=0$ so that $\operatorname{Ric}_{n,m}=\operatorname{Ric}$.
In the case of $N=\infty$, we define
\[\operatorname{Ric}_{\infty,m}:={\operatorname{Ric}}+\nabla^2v.\]
Also, we define the Witten Laplacian
\[
\Delta_{m}u:=\Delta_g u-\langle\nabla v,\nabla u\rangle_g.
\]
\subsection{Synthetic Ricci curvature bounds}
The study of synthetic Ricci curvature lower bounds for metric measure spaces originates from the works of Bakry--\'Emery \cite{BE}, Lott--Villani \cite{LV} and Sturm \cite{St1,St2}. These are the so-called $\mathsf{CD}(K,N)$ spaces ($K\in\mathbb{R}$, $1\le N\le \infty$), spaces satisfying $(K,N)$-curvature-dimension condition. To rule out non-Riemannian Finsler manifolds, Ambrosio--Gigli--Savar\'e \cite{AGS1} and Gigli \cite{Gi} add an infinitesimal Hilbertianity condition to $\mathsf{CD}(K,N)$ by introducing the class of $\mathsf{RCD}(K,N)$ spaces, spaces satisfying $(K,N)$-Riemannian-curvature-dimension condition. Since then, many geometric and analytic results of Ricci curvature lower bounds have been established on $\mathsf{CD}(K,N)$ and $\mathsf{RCD}(K,N)$ spaces. For example, Bishop--Gromov comparison theorem \cite{St2}, Cheeger--Gromoll splitting theorem \cite{Gi1}, Li-Yau estimates for heat flow \cite{BBG,GM,ZhZh}, sharp Sobolev inequality \cite{Pr}, and Levy--Gromov isoperimetric inequality \cite{CM}, etc.

It is natural to ask whether there is a synthetic notion of Ricci curvature upper bound. This question for $\operatorname{Ric}_{\infty,m}$ is addressed by Sturm in his paper \cite{St}. Sturm's definition is inspired by the fact that for any infinitesimally Hilbertian space $(X,d,m)$ the condition $\mathsf{RCD}(K,\infty)$ is equivalent to the inequality
\begin{align}\label{contract}W(\mathcal{H}_t\delta_x,\mathcal{H}_t\delta_y)\leq {\rm e}^{-Kt}d(x,y)
\end{align}
for all $x,y\in X$ and $t>0$ (see, e.g., \cite{AGS1}). Here, $W$ is the 2-Wasserstein distance on $P_2(X)$ and $\mathcal{H}_t$ is the heat flow or the gradient flow of the entropy functional $\operatorname{Ent}_m$, while $\delta_x$ is the Dirac measure at $x$ (see Section~\ref{ricupp} for definitions). The basic idea is to replace ``$\leq$'' by ``$\geq$'' in (\ref{contract}).
Solving $K$ from $W(\mathcal{H}_t\delta_x,\mathcal{H}_t\delta_y)\geq {\rm e}^{-Kt}d(x,y)$, we define the quantities
\begin{gather*}\theta^{+}(x,y)=-\liminf_{t\to 0+}\frac 1t\log(W(\mathcal{H}_t\delta_x,\mathcal{H}_t\delta_y)/d(x,y)),\\
 \theta^*(x)=\limsup_{y,z\to x}\theta^+(y,z).\end{gather*}
If $(M,g,m)$ is a weighted Riemannian manifold, Sturm~\cite{St} proved that
\[ \theta^* (x)=\sup\big\{\operatorname{Ric}_{\infty,m}( \xi, \xi)/|\xi|^2\colon \xi\in T_xM,\,\xi\neq0\big\}\] for all $x\in M$ (see Section~\ref{ricupp} for more details). This motivates the following definition.
\begin{Definition}[Sturm \cite{St}]\label{ricupdef}
We say that the metric measure space $(X,d,m)$ has {synthetic Ricci curvature upper bound $K$} if $\theta^* (x)\leq K$ for all $x\in X$.
\end{Definition}

However, it should be noted that due to the theorems of Gao--Yau \cite{GY} and Lohkamp \cite{Lo}, Ricci curvature upper bounds do not place any topological restriction on the manifold. More precisely, for any integer $n\geq 3$, there exist two positive constants $\Lambda_1=\Lambda_1(n)$, $\Lambda_2=\Lambda_2(n)$ such that each manifold $M$ of dimension $n$ admits a complete metric $g$ with $-\Lambda_1<\operatorname{Ric}(g)<-\Lambda_2$.

Neither is the set of Riemannian manifolds with Ricci curvature upper bound precompact in Gromov--Hausdorff convergence, since the number of small balls in a large ball goes to $\infty$ as $\operatorname{Ric}\to-\infty$. Given such flexibility of Ricci curvature upper bound, it is not surprising to learn that there are not as many geometric results of Ricci curvature upper bounds as Ricci curvature lower bounds. In fact, it was pointed out by Gromov in \cite[Section~5]{Gr} that the only result widely known is Bochner's theorem, which is the main subject of this article.

\subsection{Bochner's theorem}
One of the classical results of a compact manifold with negative Ricci curvature (Ricci curvature bounded above by~0) is Bochner's theorem which states that the isometry group of the manifold must be finite (see, e.g.,~\cite{Bo}). In fact, Bochner showed that there exists no continuous group of isometry by considering the Laplacian of the square norm of the Killing vector field. This is the genesis of the famous ``Bochner technique'' which produces numerous results.

Later, several authors tried to extend Bochner's theorem by estimating the order of the isometry group by various quantities. Before we mention them, we may see from the work of Lohkamp that it is impossible to control the order of the isometry group merely in terms of dimension and Ricci curvature bounds.
\begin{Theorem}[Lohkamp \cite{Lo}]
Let $M$ be a compact $n$-dimensional manifold with $n\geq 3$ and $G$ be a subgroup of $\operatorname{Diff}(M)$, the group of diffeomorphisms of~$M$. Then $G$ is the isometry group of~$M$ for some metric $g$ with $\operatorname{Ric}(g)<0$ if and only if $G$ is finite.
\end{Theorem}

Fifty years before Bochner's result came out, Hurwitz \cite{Hur} showed that when $X$ is a Riemann surface of genus $g\geq 2$, the order of the automorphism group of $X$, $\#\operatorname{Aut}(X)\leq 84(g-1)$. Later, the estimate of the order of the isometry group was generalized to hyperbolic manifolds by Huber~\cite{Hu}, to manifolds with sectional curvature bounded above from~0 by Im~Hof~\cite{Im}, to manifolds with non-positive sectional curvature and Ricci curvature negative at some point by Maeda~\cite{Ma} and to manifolds with non-positive sectional curvature and finite volume by Yamaguchi~\cite{Ya}. For general compact Riemannian manifolds with negative Ricci curvature, Katsuda~\cite{Ka} estimated the order of the isometry group by sectional curvature, dimension, diameter and injectivity radius. In~\cite{DSW}, Dai--Shen--Wei estimated the order of the isometry group by Ricci bounds, dimension, volume, and injectivity radius. Recently, Katsuda--Kobayashi~\cite{KK} gave a bound of the order of the isometry group for manifolds with negative Ricci curvature except for small portions.

There are also other generalizations of Bochner's theorem. In~\cite{Ro}, Rong showed that compact manifolds with negative Ricci curvature do not admit non-trivial invariant $F$-structure which includes the Killing vector field as a special case. Bagaev and Zhukova \cite{BZ} extended Bochner's theorem to Riemannian orbifolds with negative Ricci curvature. From the works of Deng--Hou \cite{DH} and Zhong--Zhong \cite{ZZ}, we know that a compact Finsler manifold with negative Ricci curvature has a finite isometry group. Van Limbeek~\cite{VL} estimated the order of the isometry group for manifolds on which circle actions do not exist. The list above is far from complete and can go on and on.

Our first main result is a generalization of Bochner's theorem to compact metric measure spaces with synthetic negative Ricci curvature.

\begin{Theorem}\label{mainthm}
Let $(X,d,m)$ be a compact metric measure space with $\theta^*(x)<0$ for all $x\in X$. Then $\operatorname{Iso}(X,d,m)$ is finite.
\end{Theorem}

As a corollary, we have Bochner's theorem for weighted Riemannian manifolds:

\begin{Corollary}
Let $(M,g,m)$ be a closed weighted Riemannian manifold with $\infty$-Bakry--\'Emery Ricci tensor $\operatorname{Ric}_{\infty,m}<0$, then $\operatorname{Iso}(M,g,m)$ is finite.
\end{Corollary}
This corollary can also be obtained by considering the Laplacian of the square norm of divergence free Killing vector field on $(M,g,m)$. In fact, with this method, we actually have a~stronger theorem.

\begin{Theorem}\label{nricci}
Let $(M,g,m)$ be a closed $n$-dimensional weighted Riemannian manifold. If the $N$-Bakry--\'Emery Ricci tensor $\operatorname{Ric}_{N,m}<0$ for some $n\le N\le\infty$, then $\operatorname{Iso}(M,g,m)$ is finite.
\end{Theorem}
In view of this theorem, it should be expected that if there is a notion of synthetic $N$-Ricci curvature upper bound in the future, then Bochner's theorem should hold on spaces with synthetic negative $N$-Ricci curvature.

Our next result estimates $\#\operatorname{Iso}(M,g,m)$ by sectional curvature and other geometric quantities for a weighted Riemannian manifolds $M$ with negative $\infty$-Bakry--\'Emery Ricci curvature except for small portions. This is a generalization of the main theorem in \cite{Ka} and Theorem~0.2 in~\cite{KK}. We denote the sectional curvature by~$K_M$, the injectivity radius by $\mathsf{inj}_M$, and the $L^p(M,m)$ norm by $\|\cdot\|_p$, i.e., for any Borel function $f\colon M\to\mathbb R$
\[\|f\|_p=\left( \int_M|f(x)|^p\,{\rm d}m(x)\right)^{1/p}.\]

\begin{Theorem}\label{estimate}
Let $(M,g,m)$ be a closed $n$-dimensional weighted Riemannian manifold satis\-fying
\begin{gather*}
|K_M|\leq\Lambda_1,\qquad |\nabla \operatorname{Ric}_{\infty,m}|\leq\Lambda_2,\qquad m(M)\leq V,\qquad \operatorname{diam}(M)\leq D,\qquad \mathsf{inj}_M\geq i_0,\\
 N\geq n\geq 3,\qquad \operatorname{Ric}_{N,m}\geq-\Lambda_3
\end{gather*}
for some fixed positive constants $i_0$, $N$, $\Lambda_1$, $\Lambda_2$, $\Lambda_3$, $V$, $D$.
If for some $w>0$,
\begin{align}\label{ricnegative}
\|(\theta^*+w)_+\|_{n/2}<\min\left(\frac{1}{4A},\frac{w}{4B}\right),
\end{align}
where $(\theta^*+w)_+=\max\{\theta^*+w,0\}$, and $A$, $B$ are constants such that the Sobolev inequality
\begin{align}\label{sobolev}
\|f\|_{2n/(n-2)}^2\le A\|Df\|_2^2+B\|f\|_2^2
\end{align}
 holds for all $f\in W^{1,2}(M,g,m)$, then there exists a constant $L_1=L_1(n,i_0,N,\Lambda_1,\Lambda_2,\Lambda_3,V,D,w,\allowbreak A,B)$ such that $\#\operatorname{Iso}(M,g,m)\leq L_1$.
\end{Theorem}

Condition (\ref{ricnegative}) is saying that the Bakry--\'Emery Ricci curvature greater than $-w$ has small $L^{n/2}$ norm which is what we mean by ``manifolds with negative Bakry--\'Emery Ricci curvature except for small portions''.

Our final result generalizes Theorem~1.3 in~\cite{DSW} and Theorem~0.3 in~\cite{KK}. It replaces the sectional curvature in Theorem~\ref{estimate} by Ricci curvature.

\begin{Theorem}\label{estimate2}
Let $\big(M,g,m={\rm e}^{-v}\mathrm{vol}_g\big)$ be a closed $n$-dimensional weighted Riemannian manifold satisfying
\begin{gather*}
n\geq 3,\qquad |\operatorname{Ric}|\leq\Lambda_1,\qquad |\nabla \operatorname{Ric}_{\infty,m}|\leq\Lambda_2,\qquad \|{\rm e}^{-v}\|_\infty\le E,\\ \operatorname{diam}(M)\leq D, \qquad \mathsf{inj}_M\geq i_0\end{gather*}
for some fixed positive constants $i_0$, $\Lambda_1$, $\Lambda_2$, $E$, $D$.
If for some $w>0$,
\[
\|(\theta^*+w)_+\|_{n/2}<\min\left(\frac{1}{2A},\frac{w}{2B}\right),
\]
where $A,B$ are constants such that the Sobolev inequality~\eqref{sobolev} holds, then there exists a constant $L_2=L_2(n,i_0,\Lambda_1,\Lambda_2,E,D,w,A,B)$ such that $\#\operatorname{Iso}(M,g,m)\leq L_2$.
\end{Theorem}

\begin{Remark}
Theorem \ref{estimate} and Theorem \ref{estimate2} are very similar but their proofs are different. In the proof of Theorem~\ref{estimate}, we use optimal transport to prove a key lemma (Lemma \ref{lap}) which is an analogue of Lemma~4.2 in~\cite{DSW} and Proposition~2.2 in~\cite{KK}. On the contrary, the proof of Theorem \ref{estimate2} is totally differential geometric. We have essentially used Lemma~4.2 in~\cite{DSW} and Proposition~4.1 in \cite{KK} which is obtained by estimating the Jacobi fields. It would be interesting to know whether there is an estimate of the order of the measure preserving isometry group on metric measure space instead of weighted Riemannian manifold and we believe that the proof of Theorem~\ref{estimate} gives more insight in this direction.
\end{Remark}
\begin{Remark}
It should also be noted that the constant $L_1$ in Theorem~\ref{estimate} is computable while the constant~$L_2$ in Theorem~\ref{estimate2} is not because some compactness argument is used in Lemma~4.2 of~\cite{DSW} on which Theorem~\ref{estimate2} is relied.
\end{Remark}

The plan of the rest of the paper is as follows. In Section~\ref{ricupp}, we introduce synthetic Ricci curvature upper bounds. In Section~\ref{section3}, we prove our generalization of Bochner's theorem to metric measure
spaces. In Section~\ref{section4}, we give our estimates on the order of the measure preserving isometry groups.

\section{Synthetic Ricci curvature upper bounds}\label{ricupp}

Let $(X,d,m)$ be a metric measure space such that $\int_X{\rm e}^{-Cd^2(x_0,x)}\,{\rm d}m(x)<\infty$ for some $x_0\in X$ and $C>0$. We denote by $P(X)$ the set of all Borel probability measures on~$X$. We define the {2-Wasserstein distance} on $P_2(X)=\big\{\mu\in P(X)\colon \int_Xd^2(x_0,x)\,{\rm d}\mu(x)<\infty\text{ for some }x_0\in X\big\}$ by
\[W(\mu,\nu)=\left(\inf_{\pi}\int_{X\times X}d^2(x,y)\,{\rm d}\pi(x,y)\right)^{1/2},\]
where the infimum is taken among all $\pi\in P(X\times X)$ such that $\pi$ has marginals $\mu$ and $\nu$. Such~$\pi$ is called an admissible plan. The measure~$\pi$ at which the infimum is realized is called an optimal transport plan of $\mu$ and $\nu$, which always exists (see \cite[Chapter~4]{Vi}).

 We define the entropy functional $\operatorname{Ent}_m\colon P_2(X)\to\mathbb{R}\cup\{+\infty\}$ by
\[
\operatorname{Ent}_m(\mu)=
\begin{cases}
\displaystyle \int_X\rho\log\rho\, {\rm d}m, &\mu=\rho m,\\
+\infty, &\text{otherwise}.
\end{cases}
\]

Let $K\in\mathbb R$, $N>1$ be two numbers. For $t\in[0,1]$, we define the functions $\beta^{(K,N)}_t(x,y)$ on $X\times X$ by
\[
\beta^{(K,N)}_t(x,y)=
\begin{cases}
 +\infty &\text{if $K>0$ and $\alpha > \pi$},\\
\displaystyle \left(\frac{\sin(t\alpha)}{t\sin\alpha}\right)^{N-1}\ &\text{if $K>0$ and $\alpha \in[0,\pi]$},\\
 1 &\text{if $K =0$},\\
\displaystyle \left(\frac{\sinh(t\alpha)}{t\sinh\alpha}\right)^{N-1} &\text{if $K<0$},
\end{cases}
\]
where
$\alpha = \sqrt{\frac{|K|}{N -1}} d(x,y)$.

\begin{Definition}We say that a metric measure space $(X,d,m)$ satisfies {curvature dimension condition $\mathsf{CD}(K,\infty)$} if for any two measures $\mu_0$ and $\mu_1$ in $P_2(X)$, there exists some geodesic $(\mu_t)$ in $P_2(X)$ such that for all $0\le t\le 1$, we have
\[
\operatorname{Ent}_m(\mu_t)\leq(1-t) \operatorname{Ent}_m(\mu_0)+t \operatorname{Ent}_m(\mu_1)-\frac K2t(1-t)W^2(\mu_0,\mu_1).
\]
A metric measure space $(X,d,m)$ is said to satisfy {curvature dimension condition $\mathsf{CD}(K,N)$} for $1\le N<\infty$ if for any two measures $\mu_0$ and $\mu_1$ in $P_2(X)$, there exists some geodesic $(\mu_t=\rho_tm)$ in $P_2(X)$ such that for some optimal transport plan $\pi$ of $\mu_0$ and $\mu_1$ and all $0\le t\le 1$, we have
\begin{gather*}
\int_X\rho_t^{1-\frac 1N}(x)\, {\rm d}m(x)\geq (1-t)\int_{X\times X}\rho_0(x_0)^{-\frac 1N}\beta_{1-t}(x_0,x_1)^{\frac 1N}\,{\rm d}\pi(x_0,x_1)\\
\hphantom{\int_X\rho_t^{1-\frac 1N}(x)\, {\rm d}m(x)\geq}{} +t\int_{X\times X}\rho_1(x_1)^{-\frac 1N}\beta_{t}(x_0,x_1)^{\frac 1N}\,{\rm d}\pi(x_0,x_1). 
\end{gather*}
\end{Definition}

In the case of $n$-dimensional weighted Riemannian manifold, $\mathsf{CD}(K,\infty)$ is equivalent to $\operatorname{Ric}_{\infty,m}\geq K$, and $\mathsf{CD}(K,N)$ is equivalent to $N\geq n$, $\operatorname{Ric}_{N,m}\geq K$ (see, e.g., \cite{LV,St1,St2,Vi}).

Let $(X,d,m)$ be a metric measure space and $f\in L^2(X,m)$. We define the Cheeger energy of~$f$ by
\[
\mathsf{Ch}(f)=\inf\left\{\liminf_{n\to\infty}\frac{1}{2}\int_X|Df_n|^2{\rm d}m\colon \|f_n- f\|_{L^2(X,m)}\to 0, f_n\in \mathrm{Lip} (X,d)\right\},
\]
where for any function $g\colon X\to \mathbb R$, the slope $|Dg|$ is defined as
\[|Dg|(x)=\limsup_{y\to x}\frac{|g(y)-g(x)|}{d(y,x)}.\]
We also define the descending slope $|D^-g|$ by
\[|D^-g|(x)=\limsup_{y\to x}\frac{[g(y)-g(x)]^-}{d(y,x)}.\]
Being a convex lower semicontinuous function, the Cheeger energy admits a gradient flow $H_t$ on $L^2(X,m)$.

However, the Cheeger energy $\mathsf{Ch}$ is not neccessarily a quadratic form. We say that $(X,d,m)$ is infinitesimally Hilbertian if $\mathsf{Ch}$ is a quadratic form, i.e.,
\[
\mathsf{Ch}(f+g)+\mathsf{Ch}(f-g)=2 \mathsf{Ch}(f)+2 \mathsf{Ch}(g) \qquad\text{for every}\quad f, g \in L^{2}(X, m).
\]

We also define the metric gradient flow of a function $E\colon X\to\mathbb R\cup\{+\infty\}$ to be a locally absolutely continuous curve $(\mu_t)$ on $X$ such that for all $t\geq 0$ we have
\begin{align}\label{gradflow}E(\mu_0)=E(\mu_t)+\frac12\int_0^t|\mu_s^\prime|^2\,{\rm d}s+\frac12\int_0^t|D^-E|^2(\mu_s)\,{\rm d}s,\end{align}
where $|\mu_t^\prime|=\lim\limits_{s\to t}\frac{d(\mu_s,\mu_t)}{|s-t|}$.
We denote by $\mathcal{H}_t$ the metric gradient flow of $\operatorname{Ent}_m$ on~$P_2(X)$. For more details on metric gradient flow, see, e.g., \cite{AGS2}.

\begin{Definition}[Ambrosio--Gigli--Modino--Rajala \cite{AGMR}, Ambrosio--Gigli--Savar\'e \cite{AGS1}, Gigli~\cite{Gi}]
Let $K\in\mathbb R$ and $1<N\le \infty$. A metric measure space $(X,d,m)$ is said to have {$N$-Riemannian Ricci curvature} bounded below by $K$, or to satisfy $\mathsf{RCD}(K,N)$ condition, if $(X,d,m)$ satisfies $\mathsf{CD}(K,N)$ and $\mathsf{Ch}$ is a quadratic form on $L^2(X,m)$.
\end{Definition}

In $\mathsf{RCD}(K,\infty)$ spaces, heat flow behaves well.
It is the result of Ambrosio--Gigli--Savar\'e \cite{AGS,AGS1} that the gradient flow $H_t$ of $\mathsf{Ch}$ and the gradient flow $\mathcal{H}_t$ of $\operatorname{Ent}_m$ coincides and that for every $f\in L^2(X,m)\cap L^{\infty}(X,m)$, and every $x\in X$ we have
\begin{align}\label{iden}
H_tf(x)=\int_X f(z)\,{\rm d}\mathcal{H}_t\delta_x(z).
\end{align}
Moreover, the heat flow is a $\mathrm{EVI}_K$ gradient flow so that it satisfies the contraction property~(\ref{contract}).

Turning to the Ricci curvature upper bounds, we recall the definition of \[\theta^{+}(x,y)=-\liminf_{t\to 0+}\frac 1t\log(W(\mathcal{H}_t\delta_x,\mathcal{H}_t\delta_y)/d(x,y)).\] For a function $u\colon \mathbb R\to\mathbb R$, we write $\partial^-_tu:=\liminf\limits_{s\to t}(u(s)-u(t))/(s-t)$. Thus in this notation, $\theta^{+}(x,y)=-\left.\partial_t^-\right|_{t=0}\log W(\mathcal{H}_t\delta_x,\mathcal{H}_t\delta_{y})$. Let $\gamma\colon [a,b]\to M$ be a geodesic in a weighted Riemannian manifold $(M,g,m)$. We define the average $\infty$-Ricci curvature along $\gamma$ by
\[
\rho(\gamma)=\frac{1}{b-a}\int_{a}^{b} \operatorname{Ric}_{\infty,m}\left(\dot\gamma(t), \dot\gamma(t)\right) /\left|\dot\gamma(t)\right|^{2} {\rm d} t.
\]
Sturm has proved the following two sided bounds for $\theta^+(x,y)$.
\begin{Theorem}[Sturm \cite{St}]\label{theta}
Let $(M,g,m)$ be an $n$-dimensional weighted Riemannian manifold satisfying $\mathsf{RCD}(-K^\prime,N)$ condition for some finite $K^\prime\geq 0$, $N>1$. Let $x$, $y$ be non-conjugate points in~$M$, and \mbox{$\gamma\colon [0,T]\to M$} be a unit speed geodesic connecting the two points. Then we have
\begin{align}\label{ctrl}
\rho(\gamma)\leq \theta^{+}(x, y) \leq \rho(\gamma)+\sigma(\gamma) \cdot \tan ^{2}\big(\sqrt{\sigma(\gamma)} d(x, y) / 2\big),
\end{align}
where $d$ is the metric induced by $g$, and
\[\sigma(\gamma)=\max_{t\in[0,T]}\left(\sum_{i,j=1}^n|R(e_i,\dot\gamma,e_j,\dot\gamma)|^2\right)^{\frac12}\]
 for $R$ the Riemann curvature tensor and $\{e_i\}_{i=1}^n$ some parallel orthonormal frame along $\gamma$ such that $e_1=\dot\gamma$.
\end{Theorem}

Taking the endpoints $x$, $y$ infinitely close to each other in~(\ref{ctrl}), we have
\begin{Corollary}\label{cor}Let $(M,g,m)$ be a weighted Riemannian manifold satisfying $\mathsf{RCD}(-K^\prime,N)$ condition for some finite $K^\prime\geq 0$, $N>1$.
Then we have \[\theta^* (x)=\sup\big\{\operatorname{Ric}_{\infty,m}( \xi, \xi)/|\xi|^2\colon \xi\in T_xM,\, \xi\neq0\big\}.\]
\end{Corollary}
In view of the corollary, we may view $\theta^*(x)$ as a counterpart of $\operatorname{Ric}_{\infty,m}$ in metric measure space which justifies the Definition~\ref{ricupdef}.

\section[Generalization of Bochner's theorem to metric measure spaces]{Generalization of Bochner's theorem\\ to metric measure spaces}\label{section3}

In this section we will prove our first generalization of Bochner's theorem. A fundamental lemma is the following discrete version of Bochner's theorem which is a standard technique (cf.~\cite{Ka}). For any Borel map $f\colon X\to Y$ between metric spaces, $\mu$ a Borel measure on~$X$, we denote by~$f_\#\mu$ the Borel measure on~$Y$ defined by $f_\#\mu(A)=\mu\big(f^{-1}(A)\big)$ for all $A\subset Y$ Borel.

\begin{Lemma}\label{basiclem}
Let $(X,d,m)$ be a compact metric measure space with $\theta^*(x)<0$ for all $x\in X$. Then there exists $\lambda>0$ such that if $\phi\in \operatorname{Iso}(X,d,m)$ satisfies $d(x,\phi(x))\leq\lambda$ for all $x\in X$, then~$\phi$ is the identity map.
\end{Lemma}
\begin{proof}
Since for any fixed $x\in X$, $\theta^*(x)=\limsup\limits_{y,z\to x}\theta^+(y,z)<0$, there exist $\epsilon=\epsilon(x)>0$ and $\lambda_x>0$ such that for all $y,z\in B_{\lambda_x}(x)$, we have $\theta^+(y,z)<-\epsilon$. Since $X=\bigcup_x B_{\lambda_x/2}(x)$, by compactness, there exist $x_1,\dots,x_N\in X$ such that $X=\bigcup_{i=1}^N B_{\lambda_{x_i}/2}(x_i)$. Now let $\lambda=\min\{\lambda_{x_i}/2\colon 1\le i\le N\}$. We claim that this $\lambda$ is what we want.

Let $y\in X$ and $z=\phi(y)$ be such that $d(y,z)=\max\limits_{x\in X}d(x,\phi(x))\leq\lambda$. There exists $1\le i\le N$, such that $y\in B_{\lambda_{x_i}/2}(x_i)$. Since $d(y,z)\leq\lambda$, we have $d(x_i,z)\leq d(x_i,y)+d(y,z)<\lambda_{x_i}$. Thus $y,z\in B_{\lambda_{x_i}}(x_i)$ for some $1\le i\le N$. Therefore, $\theta^+(y,z)=-\partial_t^-|_{t=0}\log(W(\mathcal{H}_t\delta_y,\mathcal{H}_t\delta_z))<-\epsilon$.
Hence for $t>0$ small, we have
\begin{gather}\label{ge}
W(\mathcal{H}_t\delta_y,\mathcal{H}_t\delta_z) \geq {\rm e}^{\epsilon t/2}d(y,z).
\end{gather}
On the other hand, since $\phi$ is an isometry that preserves the measure $m$, we have $W(\phi_\#\mu,\phi_\#\nu)=W(\mu,\nu)$, $\operatorname{Ent}_m(\phi_\#\mu)=\operatorname{Ent}_m(\mu)$, $|(\phi_\#\mu_t)^\prime|=|\mu_t^\prime|$, and $|D^- \operatorname{Ent}_m|(\phi_\#\mu)=|D^- \operatorname{Ent}_m|(\mu)$. In view of the definition (\ref{gradflow}), we see that $(\phi_\#\mu_t)$ is a gradient flow of $\operatorname{Ent}_m$ if $(\mu_t)$ is. Hence, $\phi_\# (\mathcal{H}_t\delta_y)=\mathcal{H}_t(\phi_\#\delta_y)=\mathcal{H}_t\delta_{\phi(y)}=\mathcal{H}_t\delta_z$. Therefore, $({\rm id}\times\phi)_\#(\mathcal{H}_t\delta_y)$ is an admissible plan of $(\mathcal{H}_t\delta_y,\mathcal{H}_t\delta_z)$. So we get for all $t> 0$
\begin{gather}\label{le}
W^2(\mathcal{H}_t\delta_y,\mathcal{H}_t\delta_z)\leq \int_{X}d^2(x,\phi(x))\,{\rm d}\mathcal{H}_t\delta_y(x)\leq
\max_{x\in X}d^2(x,\phi(x))=d^2(y,z).
\end{gather}

Combining (\ref{ge}) and (\ref{le}), we get, for $t>0$ small, $d^2(y,z)\geq {\rm e}^{\epsilon t}d^2(y,z)$, which is impossible unless $d(y,z)=0$, i.e., $\phi$ is the identity map.
\end{proof}

Now, we can prove our first theorem. The proof is exactly the same as the second part of the proof of \cite{Ka} which is strongly inspired by~\cite{Ma}.
\begin{proof}[Proof of Theorem \ref{mainthm}]
Take $\lambda$ as in Lemma \ref{basiclem} and $a=\lambda/4$. By compactness, let $x_1,\dots,x_N\allowbreak \in X$ be such that $X=\bigcup_{i=1}^N B_{a}(x_i)$. We define a map $F$ from the measure preserving isometry group $\operatorname{Iso}(X,d,m)$ to the symmetric group $S_N$ of degree $N$ by $F(\phi)\colon i\to j(i)$ where $j(i)$ is the smallest $j$ such that $\phi(x_i)\in B_a(x_j)$. We claim that $F$ is injective. Assume $F(\phi)=F(\psi)=j(\cdot)$. For any $x\in X$, we have $x\in B_a(x_i)$ for some $i$, then
\begin{align*}
d(\phi(x),\psi(x))&\leq d(\phi(x),\phi(x_i))+d(\phi(x_i),x_{j(i)})+d(x_{j(i)},\psi(x_i))+d(\psi(x_i),\psi(x))\\
&\leq 4a=\lambda.
\end{align*}
By Lemma \ref{basiclem}, $\phi=\psi$. Hence $F$ is injective and we get
\begin{gather}\label{fac}\# \operatorname{Iso}(X,d,m)\leq \#S_N= N!.\end{gather}This finishes the proof.
\end{proof}

Next we prove Theorem \ref{nricci}.
\begin{proof}[Proof of Theorem \ref{nricci}]Being a closed subgroup of the isometry group $\operatorname{Iso}(M,g)$, the measure preserving isometry group $\operatorname{Iso}(M,g,m)$ is a compact Lie group. (In general, the measure preserving isometry groups of $\mathsf{RCD}(K,N)$ spaces are Lie groups. See \cite{GS,So}.) So it suffices to show that its Lie algebra is of dimension $0$.
Let $\xi$ be a vector field on $M$ such that the flow $\phi_t$ of $\xi$ preserves $g$ and $m$, i.e., $(\phi_t)^*g=g$, $(\phi_t)_\#m=m$. Taking Lie derivatives, we have for all vector fields $U,W$ on $M$
\begin{gather}\langle \nabla_U\xi,W\rangle+\langle\nabla_W\xi,U\rangle=0,\label{antisymmetric}\\
\operatorname{div}_m\xi=\operatorname{div}_g\xi-\langle\nabla V, \xi\rangle=0,\label{divergencefree}\end{gather}
where $\langle\,,\,\rangle$ denotes $g$ and $\operatorname{div}_g$ is the divergence operator associated with~$g$. By (\ref{antisymmetric}), $\nabla\xi$ is antisymmetric, so $\operatorname{div}_g\xi=\operatorname{tr}(\nabla\xi)=0$, which by (\ref{divergencefree}) implies $\langle\nabla V,\xi\rangle=0$. Hence we obtain
\[\frac12\big\langle\nabla V,\nabla|\xi|^2\big\rangle=\langle\nabla_{\nabla V}\xi,\xi\rangle=-\langle\nabla_\xi\xi,\nabla V\rangle=-\nabla_\xi\langle\xi,\nabla V\rangle+\langle\nabla_\xi\nabla V,\xi\rangle=\nabla^2V(\xi,\xi).\]
By Bochner's formula for Killing vector field (see, e.g., \cite[Lemmas~2 and~3]{Bo}), we have
\[\frac12\Delta_g|\xi|^2=|\nabla\xi|^2-\operatorname{Ric}(\xi,\xi).\]
Hence we get
\[\frac12\Delta_m|\xi|^2=|\nabla\xi|^2-\operatorname{Ric}_{\infty,m}(\xi,\xi).\]
Noticing that $dV\otimes dV(\xi,\xi)=\langle\nabla V,\xi\rangle^2=0$,
we have, for all $n\le N\le\infty$,
\[\frac12\Delta_m|\xi|^2=|\nabla\xi|^2-\operatorname{Ric}_{N,m}(\xi,\xi).\]
Integrating the above equality over $M$, we obtain
\[\int_M \operatorname{Ric}_{N,m}(\xi,\xi)\,{\rm d}m=\int_M|\nabla\xi|^2\,{\rm d}m\geq 0.\]
Since $\operatorname{Ric}_{N,m}<0$, we get $\xi=0$.
\end{proof}

\section[Estimating the order of the measure preserving isometry group]{Estimating the order of the measure\\ preserving isometry group}\label{section4}

In this section, we give an explicit dependence of the upper bound in~(\ref{fac}). We will be discussing in the context of a compact weighted Riemannian manifold $(M,g,m)$. First we introduce some notations.

For any continuous map $\phi\colon M\to M$, we set $d_\phi(x)=d(x,\phi(x))$ and $\delta_\phi=\max\limits_{x\in M}d(x,\phi(x))$. If $\phi\colon M\to M$ is such a map that for all $x\in M$, $x$ and $\phi(x)$ are connected by a unique minimizing geodesic $\gamma$ (this is always true if $\delta_\phi < \mathsf{inj}_M$), then we define
\[\rho_{\phi}(x)=\frac{1}{d_\phi(x)}\int_0^{d_\phi(x)}\theta^*(\gamma(t) )\,{\rm d}t.\]

\begin{Lemma}\label{lap}
Let $(M,g,m)$ be a closed $n$-dimensional weighted Riemannian manifold with injectivity radius $\mathsf{inj}_M\geq i_0$ and sectional curvature $|K_M|\leq\Lambda$. Let $\phi\colon M\to M$ be a measure preserving isometry. Then there exist positive constants $\delta_0=\delta_0(i_0,\Lambda,n)$, $C_1=C_1(\Lambda,n)$, $C_2=C_2(\Lambda,n)$, such that if $\delta_\phi\le \delta$ then
\begin{gather}\label{lapdphi}
\Delta_m d_\phi^2\geq-2d_\phi^2\rho_\phi- C_1\tan^2(C_2d_\phi) d_\phi^2.
\end{gather}
\end{Lemma}

\begin{proof}
Let $\phi\colon M\to M$ be a measure preserving isometry.
With the same argument as (\ref{le}), we have
\[W^2(\mathcal{H}_t\delta_x,\mathcal{H}_t\delta_{\phi(x)})\leq \int_{X}d^2(z,\phi(z))\,{\rm d}\mathcal{H}_t\delta_x(z)=H_td_\phi^2(x), \]
where we have used the formula (\ref{iden}) in the last equality. Hence for $t>0$, we have
\[
\frac{W^2(\mathcal{H}_t\delta_x,\mathcal{H}_t\delta_{\phi(x)})-d_\phi^2(x)}{t}\leq\frac{H_td_\phi^2(x)-d_\phi^2(x)}{t}.
\]
Since
\[\lim_{t\to 0} W^2(\mathcal{H}_t\delta_x,\mathcal{H}_t\delta_{\phi(x)})=d_\phi^2(x) \qquad\text{and}\qquad\lim_{t\to 0}H_td_\phi^2(x)=d_\phi^2(x),\]
we obtain
\[\left.\partial_t^-\right|_{t=0}W^2(\mathcal{H}_t\delta_x,\mathcal{H}_t\delta_{\phi(x)})\leq\left.\partial_t^-\right|_{t=0}H_td_\phi^2(x).\]
Thus
\[2d_\phi(x)\left.\partial_t^-\right|_{t=0}W(\mathcal{H}_t\delta_x,\mathcal{H}_t\delta_{\phi(x)})\leq\Delta_m d_\phi^2(x).\]
By definition, we have
\[\theta^+(x,\phi(x))=-\frac{\left.\partial_t^-\right|_{t=0}W(\mathcal{H}_t\delta_x,\mathcal{H}_t\delta_{\phi(x)})}{d(x,\phi(x))},\]
so we get
\begin{align}\label{ob}
-2d_\phi^2(x)\theta^+(x,\phi(x))=2d_\phi(x)\left.\partial_t^-\right|_{t=0}W(\mathcal{H}_t\delta_x,\mathcal{H}_t\delta_{\phi(x)})\leq\Delta_m d_\phi^2(x).
\end{align}
Since $M$ is compact, it automatically satisfies some $\mathsf{RCD}(-K^\prime,N)$ condition. Choosing $\delta_0\le \frac{i_0}{2}$ and applying the upper bound estimate in (\ref{ctrl}) to $x,\phi(x)$, we have
\[
\theta^{+}(x, \phi(x)) \leq
\rho_\phi(x)+\sigma(\gamma) \cdot \tan ^{2}\big(\sqrt{\sigma(\gamma)} d_\phi(x) / 2\big).
\]
Since $|K_M|\leq \Lambda$, we have $|R(X,Y,X,Y)|\leq \Lambda$ for all $X\perp Y$ and $|X|=|Y|=1$. Let~$\{e_i\}_{i=1}^n$ be an orthonormal frame with $e_1=\dot\gamma$. Then we have $|R(e_i,e_j,e_i,e_j)|\leq \Lambda$ and $\big|R((e_i+e_j)/\sqrt 2,e_k,(e_i+e_j)/\sqrt 2,e_k)\big|\leq \Lambda$ for distinct $i$, $j$, $k$. It follows that $|R(e_i,\dot\gamma,e_j,\dot\gamma)|\leq 2\Lambda$. Observing that $R(e_i,\dot\gamma,e_j,\dot\gamma)=0$ if $i=1$ or $j=1$, we have
\begin{align*}
\sigma(\gamma)&=\max_{t\in[0,T]}\left(\sum_{i,j=1}^n|R(e_i,\dot\gamma,e_j,\dot\gamma)|^2\right)^{\frac12}=\max_{t\in[0,T]}\left(\sum_{i,j=2}^n|R(e_i,\dot\gamma,e_j,\dot\gamma)|^2\right)^{\frac12}\\
&\leq \big((n-1)^24\Lambda^2\big)^{\frac12}=2(n-1)\Lambda.
\end{align*}
Therefore if $\delta_\phi\le \delta_0=\min\big\{\frac{i_0}{2},\frac{\pi}{2C_2}\big\}$ where $C_2=\sqrt{(n-1)\Lambda/2}$, we have
\[\theta^{+}(x, \phi(x))\leq \rho_\phi(x)+2(n-1)\Lambda\tan^2\big(d_\phi\sqrt{(n-1)\Lambda/2}\big).\]
Combining~(\ref{ob}), we get
\[
\Delta_m d_\phi^2\geq-2d_\phi^2\rho_\phi-C_1\tan^2(C_2d_\phi) d_\phi^2,
\]
where $C_1=4(n-1)\Lambda$.
\end{proof}

The following lemma is a quantitative version of Lemma \ref{basiclem}.

\begin{Lemma}\label{estimatethm} Let $(M,g,m)$ be a closed weighted Riemannian manifold of dimension $n\geq 3$ with
\[|K_M|\leq\Lambda_1,\qquad |\nabla \operatorname{Ric}_{\infty,m}|\leq\Lambda_2,\qquad m(M)\leq V,\qquad \mathsf{inj}_M\geq i_0.\] If for some $w>0$
\[
\|(\theta^*+w)_+\|_{n/2}<\min\left(\frac{1}{4A},\frac{w}{4B}\right),
\]
where $A$, $B$ are constants such that the Sobolev inequality~\eqref{sobolev}
 holds, then there exists a constant $\delta=\delta(n,i_0,\Lambda_1,\Lambda_2,V,D,w,A,B)$ such that if $\delta_\phi\le\delta$ then $\phi={\rm id}$.
\end{Lemma}

\begin{proof}The argument we use here is very close to the proof of \cite[Theorem 2.1]{KK}. First suppose $\delta_\phi\le \delta\le \min\big\{\delta_0,\frac{1}{C_2}\arctan\sqrt{\frac{w}{2C_1}}\big\}$, where $\delta_0$, $C_1$, $C_2$ are as in Lemma~\ref{lap} so that (\ref{lapdphi}) holds and $C_1\tan^2(C_2\delta_\phi)\leq \frac w2$.
Multiplying both sides of~(\ref{lapdphi}) by~$d_\phi^2$ and integrating over~$M$, we get
\begin{align*}
0&\geq \int_M \big|Dd_\phi^2\big|^2{\rm d}m-\int_M\big(2\rho_\phi d_\phi^4+C_1\tan^2(C_2d_\phi) d^4_\phi\big) {\rm d}m\\
&= \big\|Dd_\phi^2\big\|_2^2-2\int_M(\rho_\phi +w)d_\phi^4{\rm d}m+\int_M\big(2w-C_1\tan^2(C_2d_\phi)\big)d_\phi^4{\rm d}m\\
&\geq \big\|Dd_\phi^2\big\|_2^2-2\|(\rho_\phi+w)_+\|_{n/2}\big\|d_\phi^2\big\|_{2n/(n-2)}^2+ \big(2w-C_1\tan^2(C_2\delta_\phi)\big)\big\|d_\phi^2\big\|_2^2\\
&\geq \big\|Dd_\phi^2\big\|_2^2-2\|(\rho_\phi+w)_+\|_{n/2}\big\|d_\phi^2\big\|_{2n/(n-2)}^2+\frac{3}{2}w\big\|d_\phi^2\big\|_2^2\\
&\geq \big\|Dd_\phi^2\big\|_2^2-2\|(\rho_\phi+w)_+\|_{n/2}\big(A\big\|Dd_\phi^2\big\|_2^2+ B\big\|d_\phi^2\big\|_2^2\big)+\frac{3}{2}w\big\|d_\phi^2\big\|_2^2,
\end{align*}
where in the last inequality we have used Sobolev inequality (\ref{sobolev}) with $f=d_\phi^2$.

We estimate the term $\|(\rho_\phi+w)_+\|_{n/2}$ by
\begin{align}
\|(\rho_\phi+w)_+\|_{n/2}&\leq\left\|\frac{1}{d_\phi(x)}\int_0^{d_\phi(x)}(\theta^*(\gamma(t) )-\theta^*(x))_+{\rm d}t\right\|_{n/2}+\|(\theta^*+w)_+\|_{n/2}\nonumber\\
&\leq \delta_\phi\Lambda_2V^{2/n}+\|(\theta^*+w)_+\|_{n/2}.\label{rhophi}
\end{align}
To get the last inequality, let $t\in[0,d_{\phi}(x)]$ be fixed. By Corollary \ref{cor}, we have
$\theta^*(\gamma(t))=\sup\big\{\operatorname{Ric}_{\infty,m}(\xi,\xi)/|\xi|^2\colon \xi\in T_{\gamma(t)}M, \xi\neq 0\big\}$. Let $\xi\in T_{\gamma(t)}M$, $|\xi|=1$ be such that $\theta^*(\gamma(t))=\operatorname{Ric}_{\infty,m}(\xi,\xi)$, and $\xi(s)$, $s\in[0,t]$ be a parallel vector field along $\gamma$ such that $\xi(t)=\xi$. Since $|\xi(0)|=1$, $\theta^*(x)\geq \operatorname{Ric}_{\infty,m}(\xi(0),\xi(0))$. Hence
\begin{align*}
(\theta^*(\gamma(t) )-\theta^*(x))_+&\leq (\operatorname{Ric}_{\infty,m}(\xi(t),\xi(t))-\operatorname{Ric}_{\infty,m}(\xi(0),\xi(0)))_+\leq t|\nabla \operatorname{Ric}_{\infty,m}|\\
&\leq \Lambda_2t\le\Lambda_2\delta_\phi .
\end{align*}

Therefore,
\begin{gather*}
0 \geq \left[1-2A\big(\delta_\phi\Lambda_2V^{2/n}+\|(\theta^*+w)_+\|_{n/2}\big)\right]\big\|Dd_\phi^2\big\|_2^2\\
\hphantom{0 \geq}{} +\left[\frac{3}{2}w-2B\big(\delta_\phi\Lambda_2V^{2/n}+\|(\theta^*+w)_+\|_{n/2}\big)\right]\big\|d_\phi^2\big\|_2^2\\
\hphantom{0}{} >\left(\frac12-2A\delta_\phi\Lambda_2V^{2/n}\right)\big\|Dd_\phi^2\big\|_2^2+\big(w-2B\delta_\phi\Lambda_2V^{2/n}\big)\big\|d_\phi^2\big\|_2^2.
\end{gather*}

If we require further that
\[\delta_\phi\le\delta=\min\left\{\delta_0,\frac{1}{C_2}\arctan\sqrt{\frac{w}{2C_1}}, \frac{1}{4A\Lambda_2V^{2/n}},\frac{w}{2B\Lambda_2V^{2/n}}\right\}\]
then $d_\phi=0$, so we get $\phi={\rm id}$.
\end{proof}

\begin{proof}[Proof of Theorem \ref{estimate}]
Since $\operatorname{Ric}_{N,m}\geq-\Lambda_3$ for some $\Lambda_3>0$ and $N\geq n\geq 3$, by generalized Bishop--Gromov theorem for $\mathsf{CD}(K,N)$ spaces (see, e.g., \cite[Theorem~2.3]{St2} or \cite[Theorem~30.11]{Vi}) the number of $\delta$-balls contained in $M$ is bounded above by a constant $L=L(D,\delta,\Lambda_3,N)$. Hence by the same argument as the proof of Theorem~\ref{mainthm} we get
\[\#\operatorname{Iso}(M,g,m)\leq L!=:L_1.\tag*{\qed}\]\renewcommand{\qed}{}
\end{proof}

Next we are going to prove Theorem \ref{estimate2}. We have the following analogue of Lemma~\ref{lap}.
\begin{Lemma}\label{lap2}
 Let $\epsilon>0$ be arbitrary and $\big(M,g,m={\rm e}^{-v}\mathrm {vol}_g\big)$ be a closed weighted Riemannian manifold of dimension $n\geq 3$ with $|\operatorname{Ric}|\leq\Lambda$, $\|{\rm e}^{-v}\|_\infty\le E$, $\operatorname{diam}(M)\leq D$ and $\mathsf{inj}_M\geq i_0$. Then there exists $\delta_1=\delta_1(\epsilon,n,\Lambda,E,D,i_0)$ such that if $\phi\colon M\to M$ is a measure preserving isometry of $M$ with $\delta_\phi\le\delta$ then we have, in the sense of distribution
\begin{gather}\label{laplacemdphi2}\Delta_md_\phi\geq -d_\phi\rho_\phi-\epsilon Gd_\phi,\end{gather}
where $G$ is a function on $M$ with a uniform $L^p(M,m)$ bound $\|G\|_p\le C(p,n,\Lambda,E,D,i_0)$ independent of $\epsilon$ for all $1\le p<\infty$.
\end{Lemma}

This lemma is a consequence of Lemma~4.2 in \cite{DSW}, which we cite here:
\begin{Lemma}[Dai--Shen--Wei \cite{DSW}]\label{dswlem}
Let $\epsilon>0$ be arbitrary and $M$ be a Riemannian manifold satisfying $|\operatorname{Ric}|\leq\Lambda$, $\mathsf{inj}_M\geq i_0$, and $\mathrm {vol}(M)\leq V$. Then there exists a $\delta_2=\delta_2 (\epsilon, n, \Lambda, i_{0}, V )$ such that for any isometry $\phi\colon M\to M$ with $\delta_{\phi} \leq \delta$, we have, in the sense of distribution
\begin{align}\label{laplacedphi2}\Delta_g d_\phi\geq-\int_0^{d_\phi}\operatorname{Ric}(\dot\gamma)\,{\rm d}t-\epsilon G d_\phi,\end{align}
where $\gamma$ is the unique minimizing geodesic connecting $x$ and $\phi(x)$ and $G$ is a function on $M$ with a uniform $L^p(M,\mathrm {vol})$ bound
\[
\|G\|_{L^p(M,\mathrm {vol})}\leq C_p(n,\Lambda,i_0,V)
\]
independent of $\epsilon$ for all $1 \leq p<\infty$.
Moreover, inequality \eqref{laplacedphi2} is pointwise for all $x\in M$ such that $x\neq \phi(x)$.
\end{Lemma}
\begin{Remark}
The statement of Lemma \ref{dswlem} may seem stronger than Lemma~4.2 of~\cite{DSW}, but a~close inspection on the proof of Lemma~4.2 of~\cite{DSW} shows that it actually proves the fact we state here. This fact is also used in Proposition~4.1 of~\cite{KK}.
\end{Remark}
\begin{proof}[Proof of Lemma \ref{lap2}]
Under the assumptions of Lemma~\ref{lap2}, by Bishop--Gromov comparison theorem, we have $\mathrm {vol}(M)\leq V=V(n,\Lambda,D)$. Taking $\delta_1=\delta_2(\epsilon, n, \Lambda, i_{0}, V(n,\Lambda,D))$, the condition of Lemma~\ref{dswlem} is satisfied.

Since $\|{\rm e}^{-v}\|_\infty\le E$, we have $\int_M|G|^p{\rm d}m=\int_M|G|^p {\rm e}^{-v}{\rm d}\mathrm {vol}\leq E C_p(n,\Lambda,D,i_0)^p$. Hence $\|G\|_p\le C(p,n,\Lambda,E,D,i_0)$.

Since $\operatorname{Ric}_{\infty,m}=\operatorname{Ric}+\nabla^2 v$, we have
\begin{align}
\int_0^{d_\phi}\operatorname{Ric}(\dot\gamma)\,{\rm d}t& =\int_0^{d_\phi}\operatorname{Ric}_{\infty,m}(\dot\gamma)\,{\rm d}t-\int_0^{d_\phi}\nabla^2 v(\dot\gamma,\dot\gamma)\,{\rm d}t\nonumber\\
&=\int_0^{d_\phi}\operatorname{Ric}_{\infty,m}(\dot\gamma)\,{\rm d}t-\int_0^{d_\phi}\langle\nabla_{\dot\gamma} \nabla v,\dot\gamma\rangle \,{\rm d}t\nonumber\\
&=\int_0^{d_\phi}\operatorname{Ric}_{\infty,m}(\dot\gamma)\,{\rm d}t-\int_0^{d_\phi}\left(\frac{{\rm d}}{{\rm d}t}\langle\nabla v,\dot\gamma\rangle-\langle\nabla v,\nabla_{\dot\gamma}\dot\gamma\rangle \right) {\rm d}t\nonumber\\
&=\int_0^{d_\phi}\operatorname{Ric}_{\infty,m}(\dot\gamma)\,{\rm d}t -\langle\nabla v(\gamma(d_\phi)),\dot\gamma(d_\phi)\rangle+\langle \nabla v(x),\dot\gamma(0)\rangle.\label{ric}
\end{align}
On the other hand, for any $x\neq\phi(x)$ we are going to calculate $\nabla d_\phi(x)$. For the sake of clarity, we denote $r_p(x)=d(p,x)$ where $d$ is the metric on $M$, and $x,p\in M$ are two points. Then for any $x\neq\phi(x)$, $\nabla d_\phi(x)=\nabla r_{\phi(x)}(x)+\nabla (r_x\circ\phi)(x)=-\dot\gamma(0)+(\phi_*)^t\dot\gamma(d_\phi)$ where $\phi_*\colon T_{x}M\to T_{\phi(x)}M$ is the tangent map and $(\phi_*)^t\colon T_{\phi(x)}M\to T_{x}M$ is the transpose of the tangent map defined by $\langle (\phi_*)^tU,W\rangle=\langle U,(\phi_*)W\rangle$ for all $U\in T_{\phi(x)}M$ and $W\in T_xM$. Therefore, we have
\begin{align}
\langle \nabla v(x),\nabla d_\phi(x)\rangle&=\left\langle\nabla v(x),-\dot\gamma(0)+(\phi_*)^t\dot\gamma(d_\phi)\right\rangle \nonumber\\
&=\left\langle(\phi_*\nabla v)(\gamma(d_\phi)),\dot\gamma(d_\phi)\right\rangle-\left\langle\nabla v(x),\dot\gamma(0)\right\rangle\nonumber\\
&=\left\langle\nabla v(\gamma(d_\phi)),\dot\gamma(d_\phi)\right\rangle-\left\langle\nabla v(x),\dot\gamma(0)\right\rangle, \label{graddphi}
\end{align}
where in the last equality we use the fact that $\phi$ is a measure preserving isometry.
Plugging~(\ref{ric}) and~(\ref{graddphi}) into (\ref{laplacedphi2}), we get for all $x\neq\phi(x)$
\[\Delta_md_\phi(x)\geq -\int_0^{d_\phi(x)}\operatorname{Ric}_{\infty,m}(\dot\gamma)\,{\rm d}t-\epsilon G(x)d_\phi(x)\geq -d_\phi(x)\rho_\phi(x)-\epsilon G(x)d_\phi(x).\]
From this pointwise inequality, with the same argument as the proof at the very end of \cite{DSW}, we have in the sense of distribution $\Delta_md_\phi\geq -d_\phi\rho_\phi-\epsilon Gd_\phi$.
\end{proof}

\begin{proof}[Proof of Theorem \ref{estimate2}]Since our Lemma~\ref{lap2} has exactly the same form as \cite[Proposition~4.1]{KK} with the only difference in the Laplacian, changing the volume element in every integral and $L^p$ norm in the proof of Theorem~0.3 in~\cite{KK} to the measure~$m$, we get the conclusion.

For tha sake of completeness, we write down the proof here. Similar to Theorem~\ref{estimate}, it suffices to prove that under the conditions of Theorem~\ref{estimate2} there exists $\delta=\delta(n,i_0,\Lambda_1,\Lambda_2,E,D,w,A,B)\allowbreak >0$, such that for any measure preserving isometry~$\phi$ with $\delta_\phi\leq\delta$, we have $\phi={\rm id}$. Then the remaining part of the proof is the same as the proof of Theorem~\ref{estimate}.

We assume $\delta_\phi\le\delta\le\delta_1$ so that Lemma \ref{lap2} is valid. Multiplying $d_{\phi}$ to (\ref{laplacemdphi2}) and integrating over $M$ with respect to the measure $m$, we have
\begin{gather}
0\geq -\int_{M} (\Delta_m d_{\phi}+\rho_\phi d_{\phi}+\epsilon G d_{\phi} ) d_{\phi}\, {\rm d}m \nonumber\\
\hphantom{0}{} \geq \|D d_{\phi} \|_{2}^{2}-\int_{M} [(\rho_\phi+w)_{+}+\epsilon G ] d_{\phi}^{2}\,{\rm d} m+w\ \|d_{\phi} \|_{2}^{2} \nonumber\\
\hphantom{0}{} \geq \|D d_{\phi} \|_{2}^{2}- ( \|(\rho_\phi+w)_{+} \|_{n / 2}+\epsilon\|G\|_{n / 2} ) \|d_{\phi} \|_{2 n /(n-2)}^{2}+w \|d_{\phi} \|_{2}^{2} \nonumber\\
\hphantom{0}{} \geq \|D d_{\phi} \|_{2}^{2}-\big(\Lambda_{2} V^{2 / n} \delta_{\phi}+ \|(\theta^*+w)_{+} \|_{n / 2}+\epsilon\|G\|_{n / 2}\big) \|d_{\phi} \|_{2 n /(n-2)}^{2}+w \|d_{\phi} \|_{2}^{2} \label{rhophi2}\\
\hphantom{0}{} \geq \big[1-A\big(\Lambda_{2} V^{2 / n} \delta_{\phi}+ \|(\theta^*+w)_{+} \|_{n / 2}+\epsilon\|G\|_{n / 2}\big)\big]\left\|D d_{\phi}\right\|_{2}^{2}\nonumber \\ 
\hphantom{0\geq} {} +\big[w-B\big(\Lambda_{2} V^{2 / n} \delta_{\phi}+ \|(\theta^*+w)_{+} \|_{n / 2}+\epsilon\|G\|_{n / 2}\big)\big] \|d_{\phi} \|_{2}^{2}, \label{sobolev3}
\end{gather}
where in (\ref{rhophi2}) we have used (\ref{rhophi}) and in 
 (\ref{sobolev3}) we have used the Sobolev inequality (\ref{sobolev}). If we take
\begin{gather*}
\delta_{\phi}\leq \delta= \min \left\{\delta_1,\frac{1}{4A \Lambda_{2} V^{2 / n}}, \frac{w}{4B\Lambda_{2} V^{2 / n}}\right\}, \qquad \epsilon=\frac{C\left(n / 2, n, \Lambda_{1}, E,D, i_0\right)}{4} \min \left\{\frac{1}{A}, \frac{w}{B}\right\}
\end{gather*}
then we have
\[
0\geq \left(\frac{1}{2}-A\left\|(\theta^*+w)_{+}\right\|_{n / 2}\right)\left\|D d_{\phi}\right\|_{2}^{2}+\left(\frac{w}{2}-B\left\|(\theta^*+w)_{+}\right\|_{n / 2}\right)\left\|d_{\phi}\right\|_{2}^{2}.
\]
From the assumption of $w,$ the coefficients of $\|D d_{\phi}\|_{2}^{2}$, $\|d_{\phi}\|_{2}^{2}$ are positive, so we have $d_{\phi} \equiv 0$, i.e., $\phi={\rm id}$.
\end{proof}

\subsection*{Acknowledgements}
The results in this article are mainly part of the author's undergraduate thesis at Tsinghua University. The author would like to express his sincere gratitude to Professor Jinxin Xue who brought him into this field and gave him expert advice. He would also like to thank Professors Yann Brenier and Francois Bolley for their email of discussion and Professor Tapio Rajala for telling him the articles~\cite{GS,So} on the measure preserving isometry groups of~$\mathsf{RCD}$ spaces. Finally, he would like to thank the anonymous referees for their useful comments which leads to Theorem~\ref{estimate2}.

\pdfbookmark[1]{References}{ref}
\LastPageEnding

\end{document}